\documentclass{article}
\usepackage{amssymb}
\usepackage{latexsym}
\usepackage{amsmath}
\usepackage{mathrsfs}
\usepackage{color}
\usepackage{CJK}
\usepackage{graphicx}

\newtheorem{theorem}{Theorem}[section]

\newtheorem{lemma}{Lemma}[section]

\newtheorem{remark}{Remark}[section]

\numberwithin{equation}{section}
\newenvironment{proof}{\medskip\par\noindent{\bf Proof.}\ }{\qquad
\raisebox{-0.5mm}{\rule{1.5mm}{4mm}}\vspace{6pt}}

\newcommand{\bbr}{\mathbb{R}}
\newcommand{\bbrn}{\mathbb{R}^3}
\newcommand{\h}{H^1(\bbrn)}

\newcommand{\bbn}{\mathbb{N}}

\begin{document}
\title
{\Large\bf On a Kirchhoff type problems with potential well and indefinite potential}%

\author{
Yuanze Wu$^{a},$\thanks{Corresponding
author. E-mail address: wuyz850306@cumt.edu.cn (Yuanze Wu)}\quad
Yisheng Huang$^{b},$\thanks{E-mail address: yishengh@suda.edu.cn(Yisheng Huang)}\quad
Zeng Liu$^{c}$\thanks{E-mail address: luckliuz@163.com(Zeng Liu).}\\%
\footnotesize$^{a}${\em  College of Sciences, China University of Mining and Technology, Xuzhou 221116, P.R. China }\\%
\footnotesize$^{b}${\em  Department of Mathematics, Soochow University, Suzhou 215006, P.R. China }\\%
\footnotesize$^{c}${\em  Department of Mathematics, Suzhou University of Science and Technology, Suzhou 215009, P.R. China}}%
\date{}
\maketitle

\date{} \maketitle

\noindent{\bf Abstract:} In this paper, we study the following Kirchhoff type problem:%
$$
\left\{\aligned&-\bigg(\alpha\int_{\bbr^3}|\nabla u|^2dx+1\bigg)\Delta u+(\lambda a(x)+a_0)u=|u|^{p-2}u&\text{ in }\bbr^3,\\%
&u\in\h,\endaligned\right.\eqno{(\mathcal{P}_{\alpha,\lambda})}%
$$
where $4<p<6$, $\alpha$ and $\lambda$ are two positive parameters, $a_0\in\bbr$ is a (possibly negative) constant and $a(x)\geq0$ is the potential well.
By the variational method, we investigate the existence of nontrivial solutions to $(\mathcal{P}_{\alpha,\lambda})$.  To our best knowledge, it is the first time that the nontrivial solution of the Kirchhoff type problem is found in the indefinite case.  We also obtain the concentration behaviors of the solutions as $\lambda\to+\infty$.%

\vspace{6mm} \noindent{\bf Keywords:} Kirchhoff type problem; Indefinite potential; Potential well; Variational method.%

\vspace{6mm}\noindent {\bf AMS} Subject Classification 2010: 35B38; 35B40; 35J10; 35J20.%

\section{Introduction}
In this paper, we will study the following Kirchhoff type problem:%
$$
\left\{\aligned&-\bigg(\alpha\int_{\bbr^3}|\nabla u|^2dx+1\bigg)\Delta u+(\lambda a(x)+a_0)u=|u|^{p-2}u&\text{ in }\bbr^3,\\%
&u\in\h,\endaligned\right.\eqno{(\mathcal{P}_{\alpha,\lambda})}%
$$
where $4<p<6$, $\alpha$ and $\lambda$ are two positive parameters, $a_0\in\bbr$ is a constant and $a(x)$ is a potential satisfying some conditions to be specified later.%

The Kirchhoff type problems in bounded domains is one of most popular nonlocal problems in the study areas of elliptic equations (cf. \cite{CKW11,CWL12,LLS14,LLT15,N14,N141,ZP06} and the references therein).  One motivation comes from the very important application to such problems in physics. Indeed, The Kirchhoff type problem in bounded domains is related to the stationary analogue of the following model:
\begin{equation}\label{eq991}
\left\{\aligned &u_{tt}-\bigg(\alpha\int_{\Omega}|\nabla u|^2dx+\beta\bigg)\Delta u=h(x,u)\quad\text{in }\Omega\times(0, T),\\
&u=0\quad\text{on }\partial\Omega\times(0, T),\\
&u(x,0)=u_0(x),\quad u_t(x,0)=u^*(x),\endaligned\right.
\end{equation}
where $T>0$ is a constant, $u_0, u^*$ are continuous functions.  Such model was first proposed by Kirchhoff in 1883 as an extension of the classical D'Alembert's wave equations for free vibration of elastic strings, Kirchhoff's model takes into account the changes in length of the string produced by transverse vibrations.  In \eqref{eq991}, $u$ denotes the displacement, $h(x,u)$ the external force and $\beta$ the initial tension while $\alpha$ is related to the intrinsic properties of the string (such as Young¡¯s modulus).  For more details on the physical background of Kirchhoff type problems, we refer the readers to \cite{A12,K83}.

The Kirchhoff type nonlocal term was introduced to the elliptic equations in $\bbr^3$ by He and Zou in \cite{HZ12}, where, by using the variational method, some existence results of the nontrivial solutions were obtained.  Since then, many papers have been devoted to such topic, see for example \cite{AF12,HLP14,LLS12,LY131,SW14,WHL15} and the references therein.  In particular, in a very recent paper \cite{SW14}, Sun and Wu have studied the following Kirchhoff type problem:
$$
\left\{\aligned&-\bigg(\mu\int_{\bbr^3}|\nabla u|^2dx+\nu\bigg)\Delta u+\lambda a(x)u=f(x,u)&\text{ in }\bbr^3,\\%
&u\in H^1(\bbr^N),\endaligned\right.%
$$
where $\mu,\nu,\lambda>0$ are parameters and $a(x)$ satisfies the following conditions:
\begin{enumerate}
\item[$(A_1)$] $a(x)\in C(\bbrn)$ and $a(x)\geq0$ on $\bbrn$.%
\item[$(A_2)$] There exists $a_\infty>0$ such that $|\mathcal{A}_\infty|<+\infty$, where $\mathcal{A}_\infty=\{x\in\bbr^3\mid a(x)<a_\infty\}$ and $|\mathcal{A}_\infty|$ is the Lebesgue measure of the set $\mathcal{A}_\infty$.%
\item[$(A_3)$] $\Omega=\text{int} a^{-1}(0)$ is a bounded domain and has smooth boundaries with $\overline{\Omega}=a^{-1}(0)$.%
\end{enumerate}
By using the variational method, they obtain some existence and non-existence results of the nontrivial solutions when $f(x,u)$ is $1$--asymptotically linear, $3$--asymptotically linear or $4$--asymptotically linear at infinity.

Under the conditions $(A_1)$--$(A_3)$, $\lambda a(x)$ is called as the steep potential well for $\lambda$ sufficiently large and the depth of the well is controlled by the parameter $\lambda$.  Such potentials were first introduced by Bartsch and Wang in \cite{BW95} for the scalar Schr\"odinger equations.  An interesting phenomenon for this kind of Schr\"odinger equations is that, one can expect to find the solutions which are concentrated at the bottom of the wells as the depth goes to infinity.  Due to this interesting property, such topic for the scalar Schr\"odinger equations was studied extensively in the past decade.  We refer the readers to \cite{BT13,DS07,LHL11,ST09,WZ09} and the references therein.  Recently, the steep potential well was also considered for some other elliptic equations and systems, see for example \cite{FSX10,GT121,JZ11,YT14,ZLZ13} and the references therein.  To our best knowledge, most of the literatures on this topic are devoted to the definite case while the indefinite case was only considered in \cite{BT13,DS07} for the the scalar Schr\"odinger equations and in \cite{ZLZ13} for the Schr\"odinger--Poisson systems.

Inspired by the above facts, we wonder what will happen for the Kirchhoff type problem with steep potential wells in the indefinite case of $a<0$?  To our best knowledge, this kind of  problems has not been studied yet in the literatures.  Thus, the purpose of this paper is to explore the preceding problems.

Before stating our results, we shall introduce some notations.  By the condition $(A_3)$, it is well known that in the case of $a_0\neq0$, all the eigenvalues $\{\gamma_i\}$ of the following problem
\begin{equation}\label{eq001}
-\Delta u=\gamma |a_0|u\quad u\in H_0^1(\Omega)
\end{equation}
satisfy $\gamma_1<\gamma_2<\gamma_3<\cdots<\gamma_i<\cdots$ with $\gamma_i\to+\infty$ as $i\to\infty$ and the multiplicity of $\gamma_i$ is finite for every $i\in\bbn$.  In particular, $\gamma_1$ is simple.  For each $i\in\bbn$, denote the corresponding eigenfunctions and the eigenspace of $\gamma_i$ by $\{\varphi_{i,j}\}_{j=1,2,\cdots,k_i}$ and $\mathcal{N}_i=$span$\{\varphi_{i,j}\}_{j=1,2,\cdots,k_i}$ respectively, where $k_i$ are the multiplicity of $\gamma_i$, then $\varphi_{i,j}$ can be chosen so that $\|\varphi_{i,j}\|_{L^2(\Omega)}=\frac{1}{|a_0|^2}$ and $\{\varphi_{i,j}\}$ can form a basis of $H_0^1(\Omega)$.  Let
\begin{equation}\label{eq9998}
k_0^*=\inf\{k\mid\gamma_k>1\},
\end{equation}
then our main result in this paper can be stated as follows.
\begin{theorem}\label{thm001}
Suppose that the conditions $(A_1)$--$(A_3)$ hold.  If either $a_0\geq0$ or $a_0<0$ with $\gamma_{k_0^*-1}<1$ then there exist positive constants $\alpha_*$ and $\Lambda_*$ such that $(\mathcal{P}_{\alpha, \lambda})$ has a nontrivial solution $u_{\alpha, \lambda}$ for all $\lambda>\Lambda_*$ and $\alpha\in(0 ,\alpha_*)$.  Moreover, $u_{\alpha,\lambda}\to u_{\alpha}$ strongly in $H^1(\bbr^3)$ as $\lambda\to+\infty$ up to a subsequence and $u_\alpha$ is a nontrivial solution of the following Kirchhoff type problem:
$$
\left\{\aligned-\bigg(\alpha\int_{\Omega}|\nabla u|^2dx+1\bigg)\Delta u+a_0 u&=|u|^{p-2}u&\quad\text{in }\Omega,\\
u&=0&\quad\text{on }\partial\Omega.\endaligned\right.\eqno{(\mathcal{P}_{\alpha}^*)}%
$$
\end{theorem}

\begin{remark}
\begin{enumerate}{\em
\item[$(a)$]  If $a_0<0$ with $|a_0|$ large enough then it is easy to see that $k_0^*>1$.  It follows that $(\mathcal{P}_{\alpha,\lambda})$ is indefinite in a suitable Hilbert space (see Lemma~\ref{lem005} for more details).  To out best knowledge, Theorem~\ref{thm001} is the first result for the Kirchhoff type problem in $\bbr^3$ for the indefinite case.
\item[$(b)$]  Theorem~\ref{thm001} also gives the existence of nontrivial solutions to $(\mathcal{P}_\alpha^*)$.  Note that $(\mathcal{P}_\alpha^*)$ is also indefinite if $a_0<0$ with $|a_0|$ large enough.  Thus, to our best knowledge, it is also the first result for the Kirchhoff type problem on bounded domains in the indefinite case.}
 \end{enumerate}
\end{remark}

Through this paper, $C$ and $C_i(i=1,2,\cdots)$ will be indiscriminately used to denote various positive constants.  $o_n(1)$ and $o_\lambda(1)$ will always denote the quantities tending towards zero as $n\to\infty$ and $\lambda\to+\infty$ respecitvely.

\section{The variational setting}
By the condition $(A_1)$, we see that for every $a_0\in\bbr$ and $\lambda>\max\{0, \frac{-a_0}{a_\infty}\}$,
\begin{equation*}
E=\{u\in D^{1,2}(\bbr^3)\mid\int_{\bbrn}a(x)u^2dx<+\infty\}
\end{equation*}
equipped with the following inner product%
\begin{equation*}
\langle u,v\rangle_{\lambda}=\int_{\bbr^3}(\nabla u\nabla v+(\lambda a(x)+a_0)^+uv)dx%
\end{equation*}
is a Hilbert space, which we will denote by $E_{\lambda}$,
where $(\lambda a(x)+a_0)^+=\max\{\lambda a(x)+a_0, 0\}$.  The corresponding norm on $E_\lambda$ is given by%
\begin{equation*}
\|u\|_{\lambda}=\bigg(\int_{\bbr^3}(|\nabla u|^2+(\lambda a(x)+a_0)^+u^2)dx\bigg)^{\frac12}.%
\end{equation*}
It follows from the H\"older inequality, the Sobolev inequality and the conditions $(A_1)$--$(A_2)$ that for every $u\in E_\lambda$ with $\lambda>\max\{0, \frac{-a_0}{a_\infty}\}$,%
\begin{eqnarray}
\int_{\bbr^3}u^2dx&=&\int_{\mathcal{A}_\infty}u^2dx+\int_{\bbr^3\backslash\mathcal{A}_\infty}u^2dx\notag\\
&\leq&|\mathcal{A}_\infty|^\frac{2}{3}S^{-1}\int_{\bbr^3}|\nabla u|^2dx+\frac{1}{a_0+a_\infty\lambda}\int_{\bbr^3}(\lambda a(x)+a_0)^+u^2dx\notag\\
&\leq&\max\{|\mathcal{A}_\infty|^\frac{2}{3}S^{-1}, \frac{1}{a_0+a_\infty\lambda}\}\int_{\bbr^3}(|\nabla u|^2+(\lambda a(x)+a_0)^+u^2)dx\notag
\end{eqnarray}
and
\begin{eqnarray}
(\int_{\bbr^3}|u|^pdx)^{\frac1p}&\leq& S_p^{-\frac{1}{2}}\bigg(\int_{\bbr^3}(|\nabla u|^2+u^2)dx\bigg)^{\frac12}\notag\\
&\leq&S_p^{-\frac{1}{2}}\sqrt{1+\max\{|\mathcal{A}_\infty|^\frac{2}{3}S^{-1}, \frac{1}{a_0+a_\infty\lambda}\}}\bigg(\int_{\bbr^3}(|\nabla u|^2+(\lambda a(x)+a_0)^+u^2)dx\bigg)^{\frac12},\notag
\end{eqnarray}
where $S$ and $S_p$ are the best Sobolev embedding constant from $D^{1,2}(\bbr^3)$ to $L^6(\bbr^3)$ and $H^1(\bbr^3)$ to $L^{p}(\bbr^3)$ respectively, that is,%
\begin{equation*}
S=\inf\{\|\nabla u\|_{L^2(\bbr^3)}^2 \mid u\in D^{1,2}(\bbr^3), \|u\|_{L^{6}(\bbr^3)}^2=1\}%
\end{equation*}
and
\begin{equation*}
S_p=\inf\{\|\nabla u\|_{L^2(\bbr^3)}^2+\|u\|_{L^2(\bbr^3)}^2 \mid u\in H^{1}(\bbr^3), \|u\|_{L^{p}(\bbr^3)}^2=1\},%
\end{equation*}
where $\|\cdot\|_{L^p(\bbr^3)}$ is the usual norm in $L^p(\bbr^3)$ for all $p\geq1$.

Let $d_\lambda=\sqrt{\max\{|\mathcal{A}_\infty|^\frac{2}{3}S^{-1}, \frac{1}{a_0+a_\infty\lambda}\}}$, then we have
\begin{equation}\label{eq0001}%
\|u\|_{L^2(\bbr^N)}\leq d_\lambda\|u\|_{\lambda}\quad\text{and}\quad \|u\|_{L^p(\bbr^3)}\leq S_p^{-\frac{1}{2}}\sqrt{1+d_\lambda^2}\|u\|_{\lambda},%
\end{equation}
which yields that $E_\lambda$ is embedded continuously into $\h$ for $\lambda>\max\{0, \frac{-a_0}{a_\infty}\}$.
Moreover, by using \eqref{eq0001}, the conditions $(A_1)$--$(A_2)$ and by following a standard argument, we can show that corresponding energy functional $J_{\alpha, \lambda}(u)$ to the Problem~$(\mathcal{P}_{\alpha,\lambda})$, given by%
\begin{eqnarray*}\label{eq0131}
J_{\alpha,\lambda}(u)=\frac\alpha4\|\nabla u\|_{L^2(\bbr^3)}^4+\frac12\int_{\bbr^3}(|\nabla u|^2+(\lambda a(x)+a_0)u^2)dx-\frac{1}{p}\|u\|_{L^{p}(\bbr^3)}^{p},%
\end{eqnarray*}
is $C^2$ in $E_\lambda$  for $\lambda>\max\{0, \frac{-a_0}{a_\infty}\}$.  For the sake of convenience, we re-write the energy functional $J_{\lambda}(u)$ by%
\begin{eqnarray*}
J_{\alpha,\lambda}(u)=\frac\alpha4\|\nabla u\|_{L^2(\bbr^3)}^4+\frac12\|u\|_\lambda^2-\frac12\mathcal{D}_{\lambda}(u,u)-\frac{1}{p}\|u\|_{L^{p}(\bbr^3)}^{p},%
\end{eqnarray*}
where $\mathcal{D}_{\lambda}(u,v)=\int_{\bbr^3}(\lambda a(x)+a_0)^-uvdx$ and $(\lambda a(x)+a_0)^-=\max\{-(\lambda a(x)+a_0), 0\}$.  In what follows, inspired by \cite{DS07,ZLZ13}, we shall make some further observations on the functional $\mathcal{D}_{\lambda}(u,u)$.

By the condition $(A_1)$, $\int_{\bbr^3}(\lambda a(x)+a_0)u^2dx\geq0$ for all $u\in E_\lambda$ with $\lambda>0$ in the case of $a_0\geq0$.  It follows that $\mathcal{D}_{\lambda}(u,u)$ is definite on $E_\lambda$ with $\lambda>0$ in the case of $a_0\geq0$.  Let us consider the case of $a_0<0$ in what follows.  Let%
\begin{equation*}
\mathcal{A}_\lambda:=\{x\in\bbr^3\mid \lambda a(x)+a_0<0\},%
\end{equation*}
then by the condition $(A_3)$, we have $\Omega\subset\mathcal{A}_\lambda$, which means that $\mathcal{A}_\lambda\not=\emptyset$ for every $\lambda>0$, and moreover, by the conditions $(A_1)$--$(A_2)$, the real number
\begin{equation*}
\Lambda_{0}:=\inf\{\lambda>0\mid|\mathcal{A}_\lambda|<+\infty\}.%
\end{equation*}
satisfies $0<\Lambda_{0}\leq\frac{-a_0}{a_\infty}$.  For $\lambda>\Lambda_{0}$, we define%
\begin{equation*}
\mathcal{F}_{\lambda}:=\{u\in E_{\lambda}\mid\text{supp}u\subset\bbr^3\backslash\mathcal{A}_\lambda\}.%
\end{equation*}
It follows from the conditions $(A_1)$--$(A_3)$ that $\mathcal{F}_{\lambda}$ is nonempty, closed and convex with $\mathcal{F}_{\lambda}\not=E_{\lambda}$.  Hence, $E_{\lambda}=\mathcal{F}_{\lambda}\oplus\mathcal{F}_{\lambda}^\perp$ and $\mathcal{F}_{\lambda}^\perp\not=\emptyset$ for $\lambda>\Lambda_{0}$ in the case of $a_0<0$, where $\mathcal{F}_{\lambda}^\perp$ is the orthogonal complement of $\mathcal{F}_{\lambda}$ in $E_{\lambda}$.
\begin{lemma}\label{lem001}
Let
\begin{equation*}
\beta(\lambda):=\inf_{u\in\mathcal{F}_{\lambda}^\perp\cap\mathbb{D}_\lambda}\|u\|_\lambda^2,
\end{equation*}
where $\mathbb{D}_\lambda:=\{u\in E_\lambda\mid \mathcal{D}_\lambda(u,u)=1\}$.
If the conditions $(A_1)$--$(A_3)$ hold, then $\beta(\lambda)$ is nondecreasing as the function of $\lambda$ on $(\Lambda_0, +\infty)$ and $\beta(\lambda)$ can be attained by some $e(\lambda)\in\mathcal{F}_{\lambda}^\perp$.  Furthermore, $(e(\lambda), \beta(\lambda))\to (\varphi_1,\gamma_1)$ strongly in $\h\times\bbr$ as $\lambda\to+\infty$ up to a subsequence.
\end{lemma}
\begin{proof}First,
thanks to the definition of $\Lambda_0$, we see that $\mathcal{D}_\lambda(u,u)$ and $\|u\|_\lambda^2$ are weakly continuous and weakly low semi-continuous on $\mathcal{F}_{\lambda}^\perp$ respecitvely.  Thus, we can use a standard argument to show that $\beta(\lambda)$ can be attained by some $e(\lambda)\in\mathcal{F}_{\lambda}^\perp\cap\mathbb{D}_\lambda$ for all $\lambda>\Lambda_0$.

\noindent Next, we will show that $\beta(\lambda)$ is nondecreasing as the function of $\lambda$ on $(\Lambda_0, +\infty)$.  Indeed, let $\lambda_1\geq\lambda_2$, then by the definition of $E_\lambda$, we have $E_{\lambda_1}=E_{\lambda_2}$ in the sense of sets.  It follows that $\mathcal{F}_{\lambda_2}\subset\mathcal{F}_{\lambda_1}$, which implies $\mathcal{F}_{\lambda_1}^{\perp}\subset\mathcal{F}_{\lambda_2}^{\perp}$.  Note that $\|u\|_{\lambda_1}^2\geq\|u\|_{\lambda_2}^2$ and $\mathcal{D}_{\lambda_1}(u,u)\leq\mathcal{D}_{\lambda_2}(u,u)$ for all $u\in E_{\lambda_1}$ by the condition $(A_1)$.  Thus, due to the definition of $\beta(\lambda_1)$ and $\beta(\lambda_2)$, we can see that $\beta(\lambda_2)\leq\beta(\lambda_1)$, that is, $\beta(\lambda)$ is nondecreasing as a functional of $\lambda$ on $(\Lambda_0, +\infty)$.

\noindent Finally, we shall prove that $(e(\lambda), \beta(\lambda))\to (\varphi_1,\gamma_1)$ strongly in $\h\times\bbr$ as $\lambda\to+\infty$ up to a subsequence.
In fact, since $\int_{\bbr^3}(\lambda a(x)+a_0)^-[e(\lambda)]^2dx=1$, it implies that
\begin{eqnarray}\label{eq9999}
\lim_{\lambda\to+\infty}\int_{\bbr^3}(a(x)+\frac{a_0}{\lambda})^+[e(\lambda)]^2dx=0.
\end{eqnarray}
Note that  $H_0^1(\Omega)\subset\mathcal{F}_{\lambda}^\perp$ for all $\lambda>\Lambda_0$ due to the condition $(A_3)$, we can easily show that $0<\beta(\lambda)\leq\gamma_{1}$ for all $\lambda>\Lambda_0$.  It follows from \eqref{eq0001} that $\{e(\lambda)\}$ is bounded in $\h$ for $\lambda$.  Without loss of generality, we assume that $e(\lambda)\rightharpoonup e^*$ weakly in $\h$ and $\beta(\lambda)\to\beta^*$  as $\lambda\to+\infty$.
By the Sobolev embedding theorem, the condition $(A_2)$ and \eqref{eq9999}, we must have $(e^*,\beta^*)\in H^1_0(\Omega)\times\bbr^+$  satisfying $e^*\equiv0$ outside $\Omega$ and $|a_0|^2\int_{\Omega}|e^*|^2dx=1$ and $(e(\lambda), \beta(\lambda))\to (e^*,\beta^*)$ strongly in $L^2(\bbr^3)\times\bbr$ as $\lambda\to+\infty$ up to a subsequence, which gives that
\begin{eqnarray}
\gamma_1&\geq&\int_{\bbr^3}(|\nabla e(\lambda)|^2+(\lambda a(x)+a_0)^+[e(\lambda)]^2)dx\notag\\%
&\geq&\int_{\Omega}|\nabla e^*|^2dx+o_\lambda(1)\notag\\%
&\geq&\gamma_1+o_\lambda(1).\label{eq1002}
\end{eqnarray}
Hence, $(e(\lambda), \beta(\lambda))\to (e^*,\gamma_1)$ strongly in $\h\times\bbr$ as $\lambda\to+\infty$ up to a subsequence and $(e^*,\gamma_1)$ satisfies
\begin{equation*}
\gamma_1=\int_{\Omega}|\nabla e^*|^2dx=\inf_{u\in H^1_0(\Omega)\backslash\{0\}}\frac{\int_{\Omega}|\nabla u|^2dx}{|a_0|^2\int_{\Omega}|u|^2dx}.
\end{equation*}
Thus, $e_\alpha^*=\varphi_{1}$ and we  complete the proof.
\end{proof}

We re-denote the above  $(e(\lambda), \beta(\lambda))$ by $(e_{1}(\lambda), \beta_{1}(\lambda))$ and define
\begin{equation*}
\mathcal{F}_{\lambda,1}^\perp:=\bigg\{u\in\mathcal{F}_{\lambda}^\perp\mid\frac{\|u\|_\lambda^2}{\mathcal{D}_\lambda(u,u)}=\beta_{1}(\lambda)\bigg\}.
\end{equation*}
Since $\gamma_{2}>\gamma_{1}\geq\beta_{1}(\lambda)$ for $\lambda>\Lambda_0$, it is easy to see that $\mathcal{F}_{\lambda,1}^\perp\not=\mathcal{F}_{\lambda}^\perp$.  Thus, we have $\mathcal{F}_{\lambda}^\perp=\mathcal{F}_{\lambda,1}^\perp\oplus\mathcal{F}_{\lambda,1}^{\perp,*}$, where $\mathcal{F}_{\lambda,1}^{\perp,*}$ is the orthogonal complement of $\mathcal{F}_{\lambda,1}^\perp$ in $\mathcal{F}_{\lambda}^\perp$.
\begin{lemma}\label{lem002}
Suppose that the condition $(A_1)$--$(A_3)$ hold.  Then there exists $\Lambda_{1}\geq\Lambda_0$ such that $\mathcal{F}_{\lambda,1}^{\perp}=\text{span}\{e_{1}(\lambda)\}$ and $\beta_{2}(\lambda)$ can be attained by some $e_{2}(\lambda)\in \mathcal{F}_{\lambda,1}^{\perp,*}$ for $\lambda>\Lambda_{1}$, where
\begin{equation*}
\beta_{2}(\lambda):=\inf_{\mathcal{F}_{\lambda,1}^{\perp,*}\cap\mathbb{D}_\lambda}\|u\|_\lambda^2.
\end{equation*}
Furthermore, $(e_{2}(\lambda),\beta_{2}(\lambda))\to(\varphi_{2,j},\gamma_2)$ strongly in $\h\times\bbr$ as $\lambda\to+\infty$ up to a subsequence for some $j\in\bbn$ with $1\leq j\leq k_2$.
\end{lemma}
\begin{proof}
Since $\mathcal{D}_\lambda(u,u)$ and $\|u\|_\lambda^2$ are weakly continuous and weakly low semi-continuous on $\mathcal{F}_{\lambda,1}^{\perp,*}$ respectively, by the fact that $\mathcal{F}_{\lambda,1}^{\perp,*}$ is weakly closed for $\lambda>\Lambda_0$, we can also use a standard argument to show that $\beta_{2}(\lambda)$ can be attained by some $e_{2}(\lambda)\in \mathcal{F}_{\lambda,1}^{\perp,*}$ for $\lambda>\Lambda_0$.  For the sake of clarity, the remaining proof will be performed through the following steps.

{\bf Step~1}\quad We prove that there exists $\Lambda_{1}\geq\Lambda_0$ such that $\mathcal{F}_{\lambda,1}^{\perp}=\text{span}\{e_{1}(\lambda)\}$ for $\lambda>\Lambda_{1}$.

Indeed, suppose on the contrary that there exist $e_{1}^*(\lambda_n),e_{1}^0(\lambda_n)\in\mathcal{F}_{\alpha,\lambda,1}^{\perp}$ with
\begin{equation*}
\langle e_{1}^*(\lambda_n), e_{1}^0(\lambda_n)\rangle_{E_{\lambda_n},E_{\lambda_n}}=0
\end{equation*}
and
$$
\int_{\bbr^3}(\lambda_n a(x)+a_0)^-[e_{1}^*(\lambda_n)]^2dx=\int_{\bbr^3}(\lambda_n a(x)+a_0)^-[e_{1}^0(\lambda_n)]^2dx=1
$$
for $\{\lambda_n\}$ satisfying $\lambda_n\to+\infty$ as $n\to\infty$.  By Lemma~\ref{lem001}, we can see that $e_{1}^*(\lambda_n)\to\varphi_1$ and $e_{1}^0(\lambda_n)\to\varphi_1$ strongly in $\h$ as $n\to\infty$ up to a subsequence.  It follows from \eqref{eq1002} and Lemma~\ref{lem001} once more that
\begin{eqnarray}
2\gamma_1&=&2\beta_{1}(\lambda_n)+o_n(1)\notag\\
&=&\|e_{1}^*(\lambda_n)\|_{\lambda_n}^2+\|e_{1}^0(\lambda_n)\|_{\lambda_n}^2+o_n(1)\notag\\
&=&\|\nabla (e_{1}^*(\lambda_n)-e_{1}^0(\lambda_n))\|_{L^2(\bbr^3)}^2+o_n(1)\notag\\
&=&o_n(1),\label{eq100}
\end{eqnarray}
which is a contradiction.

{\bf Step~2}\quad We show that $\limsup_{\lambda\to+\infty}\beta_{2}(\lambda)\leq\gamma_2$.

In fact, by Step~1, we have $\varphi_{2,1}=d_{\lambda}^* e_{1}(\lambda)+\varphi_{2,1,\lambda}^*$, where $d_{\lambda}^*$ is a constant and $\varphi_{2,1,\lambda}^*$ is the projection of $\varphi_{2,1}$ in $\mathcal{F}_{\lambda,1}^{\perp,*}$.  Thus, $\langle e_{1}(\lambda),\varphi_{2,1}\rangle_{E_\lambda,E_\lambda}=d_{\lambda}^*\|e_{1}(\lambda)\|_\lambda^2$.  It follows from the condition $(A_3)$ and Lemma~\ref{lem001} that $d_{\lambda}^*\to0$ as $\lambda\to+\infty$ up to a subsequence.  Now, by the definition of $\beta_{2}(\lambda)$, we can see from Lemma~\ref{lem001}, \eqref{eq1002} and a variant of the Lebesgue dominated convergence theorem (cf. \cite[Theorem~2.2]{PK74}) that
\begin{eqnarray*}
\limsup_{\lambda\to+\infty}\beta_{2}(\lambda)&\leq&\limsup_{\lambda\to+\infty}\frac{\|\varphi_{2,1,\lambda}^*\|_\lambda^2}
{\mathcal{D}_\lambda(\varphi_{2,1,\lambda}^*,\varphi_{2,1,\lambda}^*)}\\
&=&\limsup_{\lambda\to+\infty}\frac{\|\varphi_{2,1}-d_{\lambda}^* e_{1}(\lambda)\|_\lambda^2}{\mathcal{D}_\lambda(\varphi_{2,1}-d_{\lambda}^* e_{1}(\lambda),\varphi_{2,1}-d_{\lambda}^* e_{1}(\lambda))}\\
&=&\frac{\|\nabla\varphi_{2,1}\|_{L^2(\bbr^3)}^2}{|a_0|^2\|\varphi_{2,1}\|_{L^2(\bbr^3)}^2}\\
&=&\gamma_2.
\end{eqnarray*}

{\bf Step~3}\quad We prove that $\limsup_{\lambda\to+\infty}\beta_{2}(\lambda)\geq\gamma_2$ and $(e_{2}(\lambda),\beta_{2}(\lambda))\to(\varphi_{2,j},\gamma_2)$ strongly in $\h\times\bbr$ as $\lambda\to+\infty$ up to a subsequence for some $j\in\bbn$ with $1\leq j\leq k_2$.

Actually, by Step~2, we know that $\{e_{2}(\lambda)\}$ is bounded in $D^{1,2}(\bbr^3)$.  Similarly as in the proof of Lemma~\ref{lem001}, we can see that $(e_{2}(\lambda),\beta_{2}(\lambda))\to(e_{2}^*,\beta_{2}^*)$ strongly in $L^2(\bbr^3)\times\bbr$ as $\lambda\to+\infty$ up to a subsequence with $e_{2}^*\in H^1_0(\Omega)$ and $e_{2}^*\equiv0$ outside $\Omega$.  Since $\mathcal{D}_\lambda(u,u)$ is weakly continuous on $\mathcal{F}_{\lambda}^\perp$, we also have $|a_0|^2\int_{\Omega}|e_2^*|^2dx=1$.  Furthermore, by the theory of Lagrange multipliers, we can also see that $(e_{2}^*,\beta_{2}^*)$ satisfies \eqref{eq001}.  It follows from a variant of the Lebesgue dominated convergence theorem (cf. \cite[Theorem~2.2]{PK74}) that
\begin{eqnarray*}
\|\nabla e_{2}^*\|_{L^2(\bbr^3)}^2&=&\beta_{2}^*|a_0|\|e_{2}^*\|_{L^2(\bbr^3)}^2\\
&=&\beta_{2}(\lambda)\mathcal{D}_\lambda(e_{2}(\lambda),e_{2}^*(\lambda))+o_\lambda(1)\\
&=&\int_{\bbr^3}|\nabla e_{2}(\lambda)|^2dx+o_\lambda(1)\\
&\geq&\|\nabla e_{2}^*\|_{L^2(\bbr^3)}^2.
\end{eqnarray*}
Thus, $(e_{2}(\lambda),\beta_{2}(\lambda))\to(e_{2}^*,\beta_{2}^*)$ strongly in $\h\times\bbr$ as $\lambda\to+\infty$ up to a subsequence.  Due to Step~2, we must have $\beta_{2}^*=\gamma_1$ or $\beta_{2}^*=\gamma_2$.  If $\limsup_{\lambda\to+\infty}\beta_{2}(\lambda)<\gamma_2$ then there exists $\{\lambda_n\}$ such that $(e_{2}(\lambda_n),\beta_{2}(\lambda_n))\to(\varphi_1,\gamma_1)$ strongly in $L^2(\bbr^3)\times\bbr$ as $n\to\infty$ up to a subsequence.  It follows from Lemma~\ref{lem001}, \eqref{eq1002} and Step~2 that
\begin{equation*}
0=\langle e_{1}(\lambda_n), e_{2}(\lambda_n)\rangle_{\lambda_n, \lambda_n}=\|\nabla\varphi_1\|_{L^2(\bbr^3)}^2+o_n(1),
\end{equation*}
which is a contradiction.
\end{proof}

Let
\begin{equation*}
\mathcal{F}_{\lambda,2}^\perp:=\bigg\{u\in\mathcal{F}_{\lambda}^\perp\mid\frac{\|u\|_\lambda^2}{\mathcal{D}_\lambda(u,u)}=\beta_{2}(\lambda)\bigg\}.
\end{equation*}
Since $\gamma_{3}>\gamma_{2}$, it yields from  Lemma~\ref{lem002} and the condition $(A_3)$ that $\mathcal{F}_{\lambda,1}^\perp\oplus\mathcal{F}_{\lambda,2}^\perp\not=\mathcal{F}_{\lambda}^\perp$.
\begin{lemma}\label{lem003}
Suppose that the conditions $(A_1)$--$(A_3)$ hold.  Then there exists $\Lambda_{2}\geq\Lambda_{1}$ such that dim$(\mathcal{F}_{\lambda,2}^\perp)\leq k_2$ for $\lambda>\Lambda_{2}$.
\end{lemma}
\begin{proof}
Let $e_{2}(\lambda),e_{2}'(\lambda)\in\mathcal{F}_{\lambda,2}^\perp$.  By Lemma~\ref{lem002}, $e_{2}(\lambda)\to\varphi_{2,j}$ and $e_{2}'(\lambda)\to\varphi_{2,j'}$ strongly in $\h\times\bbr$ as $\lambda\to+\infty$ up to a subsequence for some $j,j'\in\bbn$ with $1\leq j,j'\leq k_2$.  Clearly, two cases may occur:
\begin{enumerate}
\item[$(1)$] $\varphi_{2,j}=\varphi_{2,j'}$;
\item[$(2)$] $\varphi_{2,j}\not=\varphi_{2,j'}$ and $\int_{\Omega}\varphi_{2,j}\varphi_{2,j'}dx=0$.
\end{enumerate}
If case~$(1)$ happens then by a similar argument used in the proof of \eqref{eq100}, we can get that $\gamma_{2}=0$, which is a contradiction.  Thus, we must have case~$(2)$.  It follows that there exists $\Lambda_{2}\geq\Lambda_{1}$ such that dim$(\mathcal{F}_{\lambda,2}^\perp)\leq k_2$ for $\lambda>\Lambda_{2}$.
\end{proof}

Now, by iterating, for $m=3,4,\cdots$, we can define $\beta_{m}(\lambda)$  as follows:
\begin{equation*}
\beta_{m}(\lambda):=\inf_{\mathcal{F}_{\lambda,m}^{\perp,*}\cap\mathbb{D}_\lambda}\|u\|_\lambda^2,
\end{equation*}
where $\mathcal{F}_{\lambda,m}^{\perp,*}:=\{u\in\mathcal{F}_{\lambda}^\perp\mid\langle u,v\rangle_\lambda=0,\text{ for all }v\in\bigoplus_{i=1}^{m-1}\mathcal{F}_{\lambda,i}^\perp\}$ and
\begin{equation*}
\mathcal{F}_{\lambda,i}^\perp:=\bigg\{u\in\mathcal{F}_{\lambda}^\perp\mid\frac{\|u\|_\lambda^2}{\mathcal{D}_\lambda(u,u)}=\beta_{i}(\lambda)\bigg\}.
\end{equation*}
Similarly as Lemmas~\ref{lem002} and \ref{lem003}, we can obtain the following result.
\begin{lemma}\label{lem004}
Suppose that the condition $(A_1)$--$(A_3)$ hold.  Then there exists $\Lambda_{m}\geq\Lambda_{m-1}$ such that $\beta_{m}(\lambda)$ can be attained by some $e_{m}(\lambda)\in \mathcal{F}_{\lambda,m}^{\perp,*}$ for $\lambda>\Lambda_{m}$.
Furthermore, $(e_{m}(\lambda),\beta_{m}(\lambda))\to(\varphi_{m,j},\gamma_m)$ strongly in $\h\times\bbr$ as $\lambda\to+\infty$ up to a subsequence for some $j\in\bbn$ with $1\leq j\leq k_m$ and dim$(\mathcal{F}_{\lambda,m}^\perp)\leq k_m$ for $\lambda>\Lambda_{m}$, where
\begin{equation*}
\mathcal{F}_{\lambda,m}^\perp:=\bigg\{u\in\mathcal{F}_{\lambda}^\perp\mid\frac{\|u\|_\lambda^2}{\mathcal{D}_\lambda(u,u)}=\beta_{m}(\lambda)\bigg\}.
\end{equation*}
\end{lemma}

Let $k_0^*$ be given in \eqref{eq9998}, then by Lemmas~\ref{lem001}, \ref{lem002} and \ref{lem004}, $\bigoplus_{i=1}^{k_0^*-1}\mathcal{F}_{\lambda,i}^\perp$ and $\mathcal{F}_{\lambda,k_0^*}^{\perp,*}$ are well defined for $\lambda>\Lambda_{k_0^*}$.
\begin{lemma}\label{lem005}
Suppose that  the conditions $(A_1)$--$(A_3)$ hold.  If $\gamma_{k_0^*-1}<1$
then there exists $\widetilde{\Lambda}_{k_0^*}\geq\Lambda_{k_0^*}$ such that for $\lambda>\widetilde{\Lambda}_{k_0^*}$, it holds that
\begin{enumerate}
\item[$(1)$] $\|u\|_{\lambda}^2-\mathcal{D}_\lambda(u,u)\leq\frac12(1-\frac{1}{\gamma_{k_0^*-1}})\|u\|_{\lambda}^2$ in $\bigoplus_{i=1}^{k_0^*-1}\mathcal{F}_{\lambda,i}^\perp$;
\item[$(2)$] $\|u\|_{\lambda}^2-\mathcal{D}_\lambda(u,u)\geq\frac12(1-\frac{1}{\gamma_{k_0^*}})\|u\|_{\lambda}^2$ in $\mathcal{F}_{\lambda,k_0^*}^{\perp,*}$.
\end{enumerate}
\end{lemma}
\begin{proof}
The proof follows immediately from Lemmas~\ref{lem001}, \ref{lem002} and \ref{lem004}.
\end{proof}

\begin{remark}\label{rmk001}{\em
By Lemmas~\ref{lem002}--\ref{lem004}, we also have $\bigoplus_{i=1}^{k_0^*-1}\mathcal{F}_{\lambda,i}^\perp=\emptyset$ in the case of $\gamma_1>1$ while $\bigoplus_{i=1}^{k_0^*-1}\mathcal{F}_{\lambda,i}^\perp\not=\emptyset$ and dim$(\bigoplus_{i=1}^{k_0^*-1}\mathcal{F}_{\lambda,i}^\perp)\leq\sum_{i=1}^{k_0^*-1}k_i$ in the case of $\gamma_1<1$.}
\end{remark}

\section{The nontrivial solution}
We first consider the case of $a_0<0$.  Due to the decomposition of $E_\lambda$, we will find the nontrivial solution by the linking theorem.  Let us first verify that $J_{\alpha,\lambda}(u)$ has a linking structure in $E_\lambda$ in the case of $a_0<0$.
\begin{lemma}\label{lem006}
Suppose that the conditions $(A_1)$--$(A_3)$ hold and $a_0<0$.  For every $\alpha>0$, if $\beta_{k_0^*-1}<1$ then there exists $\rho>0$ independent of $\lambda$ such that
\begin{eqnarray}
\inf_{\mathcal{F}_{\lambda,k_0^*}^{\perp,*}\cap\mathbb{S}_{\lambda,\rho}}J_{\alpha,\lambda}(u)\geq d_0\label{eq102}
\end{eqnarray}
for all $\lambda>\widetilde{\Lambda}_{k_0^*}$, where $\mathbb{S}_{\lambda,\rho}:=\{u\in E_\lambda\mid\|u\|_\lambda=\rho\}$ and $d_0$ is a constant independent of $\lambda$ and $\alpha$.
\end{lemma}
\begin{proof}
By \eqref{eq0001} and Lemma~\ref{lem005}, for every $u\in \mathcal{F}_{\lambda,k_0^*}^{\perp,*}$, we have
\begin{eqnarray}
J_{\alpha,\lambda}(u)&=&\frac\alpha4\|\nabla u\|_{L^2(\bbr^3)}^4+\frac12\|u\|_\lambda^2-\frac12\mathcal{D}_{\lambda}(u,u)-\frac{1}{p}\|u\|_{L^{p}(\bbr^3)}^{p}\notag\\
&\geq&\frac14(1-\frac{1}{\gamma_{k_0^*}})\|u\|_{\lambda}^2-S_p^{-\frac{p}{2}}(1+d_\lambda^2)^{\frac p2}\|u\|_{\lambda}^p\notag\\
&\geq&\|u\|_{\lambda}^2\bigg(\frac14(1-\frac{1}{\gamma_{k_0^*}})-S_p^{-\frac{p}{2}}(1+d_\lambda^2)^{\frac p2}\|u\|_{\lambda}^{p-2}\bigg).\label{eq1001}
\end{eqnarray}
Note that $d_\lambda=\sqrt{\max\{|\mathcal{A}_\infty|^\frac{2}{3}S^{-1}, \frac{1}{a_0+a_\infty\lambda}\}}$, so that $d_\lambda\leq\sqrt{\max\{|\mathcal{A}_\infty|^\frac{2}{3}S^{-1}, \frac{1}{a_0+a_\infty\widetilde{\Lambda}_{k_0^*}}\}}$ for $\lambda>\widetilde{\Lambda}_{k_0^*}$.  It follows from \eqref{eq1001} that there exists $\rho>0$ independent on $\lambda$ such that \eqref{eq102} holds for all $\lambda>\widetilde{\Lambda}_{k_0^*}$.
\end{proof}

Let
\begin{eqnarray*}
\mathcal{Q}_{\lambda,k_0^*}:=\{u=v+te_{k_0^*}(\lambda)\mid t\geq0\text{ and }v\in\bigoplus_{i=1}^{k_0^*-1}\mathcal{F}_{\lambda,i}^\perp\}.
\end{eqnarray*}
Then we have the following.
\begin{lemma}\label{lem007}
Suppose that the conditions $(A_1)$--$(A_3)$ hold and $a_0<0$.  If $\gamma_{k_0^*-1}<1$ then there exist $\alpha_0>0$ and $R_0>\rho$ independent of $\lambda$ such that
\begin{eqnarray*}
\sup_{\partial\mathcal{Q}_{\lambda,k_0^*}^{R_0}}J_{\alpha,\lambda}(u)\leq\frac12d_0
\end{eqnarray*}
for all $\lambda>\widetilde{\Lambda}_{k_0^*}$  in the case of $\alpha\in(0, \alpha_0)$, where $d_0$ is given in lemma~\ref{lem006}, $\mathcal{Q}_{\lambda,k_0^*}^{R_0}:=\mathcal{Q}_{\lambda,k_0^*}\cap \mathbb{B}_{\lambda,R_0}$ and $\mathbb{B}_{\lambda,R_0}:=\{u\in E_\lambda\mid \|u\|_\lambda\leq R_0\}$.
\end{lemma}
\begin{proof}
Let $u_{\lambda}\in\partial\mathcal{Q}_{\lambda,k_0^*}^{R}$.  Then one of the following two cases must happen:
\begin{enumerate}
\item[$(a)$] $u_{\lambda}=R\widetilde{u}_{\lambda}$ with $\widetilde{u}_{\lambda}\in \bigoplus_{i=1}^{k_0^*-1}\mathcal{F}_{\lambda,i}^\perp$ and $\|\widetilde{u}_{\lambda}\|_\lambda\leq1$.
\item[$(b)$] $u_{\lambda}=R\widetilde{u}_{\lambda}$ with $\widetilde{u}_{\lambda}\in \mathcal{Q}_{\lambda,k_0^*}^{1}\backslash\bigoplus_{i=1}^{k_0^*-1}\mathcal{F}_{\lambda,i}^\perp$ and $\|\widetilde{u}_{\lambda}\|_\lambda=1$.
\end{enumerate}
If the case~$(b)$ happens then by Lemma~\ref{lem005}, we deduce that
\begin{eqnarray}\label{eq103}
J_{\alpha,\lambda}(u_{\lambda})=J_{\alpha,\lambda}(R\widetilde{u}_{\lambda})
\leq\frac{\alpha}{4}R^4+\frac12(1-\frac{1}{\gamma_{k_0^*}})R^2-\frac1p\|R\widetilde{u}_{\lambda}\|_{L^p(\bbr^3)}^p.
\end{eqnarray}
Since $\widetilde{u}_{\lambda}\in \bigoplus_{i=1}^{k_0^*}\mathcal{F}_{\lambda,i}^\perp$, by Lemmas~\ref{lem001}, \ref{lem002} and \ref{lem004}, $\widetilde{u}_{\lambda}=\widetilde{u}+o_\lambda(1)$ strongly in $\h$ for some $\widetilde{u}\in$span$\{\varphi_{i,j}\}^{i=1,2,\cdots,k_0^*}_{j=1,2,\cdots,k_i}$ and $\int_{\Omega}|\nabla \widetilde{u}|^2dx=1$.
Thus, $\|\widetilde{u}_{\lambda}\|_{L^p(\bbr^3)}^p=\|\widetilde{u}\|_{L^p(\bbr^3)}^p+o_\lambda(1)$ due to the Sobolev embedding theorem.  Note that
dim(span$\{\varphi_{i,j}\}^{i=1,2,\cdots,k_0^*}_{j=1,2,\cdots,k_i})\leq\sum_{i=1}^{k_0^*-1}k_i+1$ for all $\lambda>\widetilde{\Lambda}_{k_0^*}$ by Remark~\ref{rmk001}.  Therefore, there exists a constant $M>0$ such that $\|u\|_{L^p(\bbr^3)}\geq M$ for all $u\in$span$\{\varphi_{i,j}\}^{i=1,2,\cdots,k_0^*}_{j=1,2,\cdots,k_i}$ with $\int_{\Omega}|\nabla u|^2dx=1$.  In particular, $\|\widetilde{u}\|_{L^p(\bbr^3)}\geq M$.  It follows from $4<p<6$ and \eqref{eq103} that there exists a constant $R_0(>\rho)$ such that $J_{\alpha,\lambda}(R_0\widetilde{u}_{\lambda})\leq0$ for all $\lambda>\widetilde{\Lambda}_{k_0^*}$.
Now, we consider the case of $(a)$.  By Lemma~\ref{lem005} once more, we know that
\begin{eqnarray*}
J_{\alpha,\lambda}(u_{\lambda})=J_{\alpha,\lambda}(R\widetilde{u}_{\lambda})\leq\frac{\alpha}{4}R_0^4.
\end{eqnarray*}
Thus, there exists $\alpha_0>0$ such that $J_{\alpha,\lambda}(u_{\lambda})\leq \frac12 d_0$ for $\lambda>\widetilde{\Lambda}_{k_0^*}$ and $\alpha\in(0, \alpha_0)$.
\end{proof}

Due to Lemmas~\ref{lem006} and \ref{lem007}, we can see that $J_{\alpha,\lambda}(u)$ has a linking structure in $E_\lambda$ with $\lambda>\widetilde{\Lambda}_{k_0^*}$ and $\alpha\in(0, \alpha_0)$ in the case of $a_0<0$.
By the linking theorem, there exists $\{u_n\}\subset E_\lambda$ such that $(1+\|u_n\|_\lambda)J_{\alpha,\lambda}'(u_n)=o_n(1)$ strongly in $E_\lambda^*$ and
$J_{\alpha,\lambda}(u_n)=c_{\alpha,\lambda}+o_n(1)$, where $E_\lambda^*$ is the dual space of $E_\lambda$.  Furthermore, $c_{\alpha,\lambda}\in[d_0, \frac{\alpha}{4}R_0^4+\frac12(1-\frac{1}{\gamma_{k_0^*}}) R_0^2]$.  Note that in the special case $\gamma_1>1$, the linking structure is actually the mountain pass geometry.  Thus, the linking theorem can be replaced by the mountain pass theorem and we can also obtain a sequence $\{u_n\}\subset E_\lambda$ such that $(1+\|u_n\|_\lambda)J_{\alpha,\lambda}'(u_n)=o_n(1)$ strongly in $E_\lambda^*$ and
$J_{\alpha,\lambda}(u_n)=c_{\alpha,\lambda}+o_n(1)$.  In the case $a_0\geq0$, since $4<p<6$ and the fact that $\mathcal{D}_\lambda(u,u)=0$ in $E_\lambda$, by using a standard argument, we can verify that $J_{\alpha, \lambda}(u)$ has a mountain pass geometry in $E_\lambda$ for $\lambda>0$, that is,
\begin{enumerate}
\item[$(a)$] $\inf_{\mathbb{S}_{\lambda,\overline{\rho}}}J_{\alpha,\lambda}(u)\geq C$ for some $\overline{\rho}>0$;
\item[$(b)$] $J_{\alpha,\lambda}(\overline{R}_0\phi)\leq0$ for some $\overline{R}_0>\overline{\rho}$ and $\phi\in H_0^1(\Omega)$,
\end{enumerate}
which also gives the existence of a sequence $\{u_n\}\subset E_\lambda$ such that $(1+\|u_n\|_\lambda)J_{\alpha,\lambda}'(u_n)=o_n(1)$ strongly in $E_\lambda^*$ and
$J_{\alpha,\lambda}(u_n)=c_{\alpha,\lambda}+o_n(1)$ with $c_{\alpha,\lambda}\in [C_\alpha, C_\alpha']$, where $C_\alpha,C'_\alpha$ are two positive constants independent of $\lambda$.  In a word, in both cases of $a_0<0$ and $a_0\geq 0$, for $\lambda>\widetilde{\Lambda}_{k_0^*}$, there exists $\{u_n\}\subset E_\lambda$ such that $(1+\|u_n\|_\lambda)J_{\alpha,\lambda}'(u_n)=o_n(1)$ strongly in $E_\lambda^*$ and
$J_{\alpha,\lambda}(u_n)=c_{\alpha,\lambda}+o_n(1)$ with $c_{\alpha, \lambda}\in[C_\alpha, C'_\alpha]$.

\begin{lemma}\label{lem008}
Suppose that  the conditions $(A_1)$--$(A_3)$ hold.  For every $\alpha>0$, if either $a_0\geq0$ or $a_0<0$ with $\beta_{k_0^*-1}<1$ then $\{\|u_n\|_\lambda\}$ is bounded.
\end{lemma}
\begin{proof}
Since $\lambda>\widetilde{\Lambda}_{k_0^*}$, by the condition $(A_2)$ and the H\"older and the Sobolev inequalities, we obtain that
\begin{eqnarray*}
\mathcal{D}_\lambda(u_n,u_n)\leq|a_0|\int_{\mathcal{A}_\infty}|u_n|^2dx\leq|a_0||\mathcal{A}_\infty|^{\frac{2}{3}}S^{-1}\|\nabla u_n\|_{L^2(\bbr^3)}^2.
\end{eqnarray*}
Note that $(1+\|u_n\|_\lambda)J_{\alpha,\lambda}'(u_n)=o_n(1)$ strongly in $E_\lambda^*$ and
$J_{\alpha,\lambda}(u_n)=c_{\alpha,\lambda}+o_n(1)$, by the Young inequality and the fact that $4<p<6$, we deduce that
\begin{eqnarray*}
c_{\alpha,\lambda}+o_n(1)&=&J_{\alpha,\lambda}(u_n)-\frac1p\langle J_{\alpha,\lambda}'(u_n), u_n\rangle_{E_\lambda^*, E_\lambda}\\
&=&\alpha(\frac{1}{4}-\frac1p)\|\nabla u_n\|_{L^2(\bbr^3)}^4+(\frac12-\frac1p)\|u_n\|_\lambda^2-(\frac12-\frac1p)\mathcal{D}_\lambda(u_n,u_n)\\
&\geq&\frac{p-4}{4p}(\alpha\|\nabla u_n\|_{L^2(\bbr^3)}^4+\|u_n\|_\lambda^2)-\frac{p-2}{2p}|a_0||\mathcal{A}_\infty|^{\frac{2}{3}}S^{-1}\|\nabla u_n\|_{L^2(\bbr^3)}^2\\
&\geq&\frac{p-4}{8p}(\alpha\|\nabla u_n\|_{L^2(\bbr^3)}^4+\|u_n\|_\lambda^2)-\frac{2(p-2)^2}{\alpha(p-4)p}|a_0|^2|\mathcal{A}_\infty|^{\frac{4}{3}}S^{-2},
\end{eqnarray*}
where $\langle \cdot,\cdot\rangle_{E_{\lambda}^*, E_{\lambda}}$ is the duality pairing of $E_{\lambda}^*$ and $E_{\lambda}$.
The preceding inequality, together with $c_{\alpha,\lambda}\in[C_\alpha, C'_\alpha]$ and $4<p<6$, implies $\{\|u_n\|_\lambda\}$ is bounded.
\end{proof}

By Lemma~\ref{lem008}, we can see that $u_n=u_{\alpha,\lambda}+o_n(1)$ weakly in $E_\lambda$ for some $u_{\alpha,\lambda}\in E_\lambda$ up to a subsequence.  Without loss of generality, we may assume that $u_n=u_{\alpha,\lambda}+o_n(1)$ weakly in $E_\lambda$.

\begin{lemma}\label{lem009}
Suppose that the conditions $(A_1)$--$(A_3)$ hold.  For every $\alpha>0$, if either $a_0\geq0$ or $a_0<0$ with $\beta_{k_0^*-1}<1$ then there exists $\overline{\Lambda}_{k_0^*}>\widetilde{\Lambda}_{k_0^*}$ such that $u_{\alpha,\lambda}$ is a nontrivial solution of $(\mathcal{P}_{\alpha,\lambda})$ for $\lambda>\overline{\Lambda}_{k_0^*}$.
\end{lemma}
\begin{proof}
We first prove that $u_{\alpha,\lambda}\not=0$ in $E_\lambda$.  Indeed, suppose on the contrary, then by the Sobolev embedding theorem, we can see that $u_n=o_n(1)$ strongly in $L^2_{loc}(\bbr^3)$, which, together with the condition $(A_2)$, implies $u_n=o_n(1)$ strongly in $L^2(\mathcal{A}_\infty)$.  It follows from Lemma~\ref{lem008}, the conditions $(A_1)$--$(A_2)$ and the H\"older and the Sobolev inequality that
\begin{eqnarray}
\int_{\bbr^3}|u_n|^pdx&\leq&\bigg(\int_{\bbr^3}|u_n|^2dx\bigg)^{\frac{6-p}{4}}\bigg(\int_{\bbr^3}|u_n|^6dx\bigg)^{\frac{p-2}{4}}\notag\\
&\leq&S^{-\frac{3(p-2)}{4}}\|\nabla u_n\|_{L^2(\bbr^3)}^{\frac{3(p-2)}{2}}\bigg(\int_{\bbr^3\backslash\mathcal{A}_\infty}|u_n|^2dx+o_n(1)\bigg)^{\frac{6-p}{4}}\notag\\
&\leq&S^{-\frac{3(p-2)}{4}}(C_{1}+o_n(1))^{\frac{5p-10}{4}}\bigg(\frac{1}{a_0+a_\infty\lambda}\bigg)^{\frac{6-p}{4}}\|u_n\|_\lambda^2+o_n(1).\label{eq8888}
\end{eqnarray}
On the other hand, by the conditions $(A_1)$--$(A_2)$ once more, we have
\begin{eqnarray}
\mathcal{D}_\lambda(u_n, u_n)\leq|a_0|\int_{\mathcal{A}_\infty}|u_n|^2dx=o_n(1).\label{eq8889}
\end{eqnarray}
Therefore, we deduce from the fact that $(1+\|u_n\|_\lambda)J_{\alpha,\lambda}'(u_n)=o_n(1)$ strongly in $E_\lambda^*$ that
\begin{eqnarray*}
\alpha\|\nabla u_n\|_{L^2(\bbr^2)}^4+\|u_n\|_\lambda^2\leq S^{-3(p-2)}(C_{1}+o_n(1))^{\frac{5p-10}{4}}
\bigg(\frac{1}{a_0+a_\infty\lambda}\bigg)^{\frac{6-p}{4}}\|u_n\|_\lambda^2+o_n(1),
\end{eqnarray*}
which yields that there exists $\overline{\Lambda}_{k_0^*}>\widetilde{\Lambda}_{k_0^*}$ dependent of $\alpha$ such that $u_n=o_n(1)$ strongly in $E_\lambda$ with $\lambda>\overline{\Lambda}_{k_0^*}$.  It is impossible since $c_{\alpha,\lambda}\geq C_\alpha>0$ for all $\lambda>\widetilde{\Lambda}_{k_0^*}$.  Therefore $u_{\alpha,\lambda}\not=0$ in $E_\lambda$. It remains to show that $J_{\alpha,\beta}'(u_{\alpha,\beta})=0$ in $E_\lambda^*$.  In fact, without loss of generality, we may assume that $\|u_n\|_{L^2(\bbr^3)}^2=A+o_n(1)$ and consider the following energy functional
\begin{equation*}
I_{\alpha, \lambda}(u)=\frac{\alpha A}{2}\|u\|_{L^2(\bbr^3)}^2+\frac12\|u\|_\lambda^2-\frac12 \mathcal{D}_\lambda(u, u)-\frac1p\|u\|_{L^p(\bbr^3)}^p.
\end{equation*}
Clearly, by \eqref{eq0001}, $I_{\alpha, \lambda}(u)$ is of $C^2$ in $E_\lambda$ for $\lambda>\overline{\Lambda}_{k_0^*}$.  Since $(1+\|u_n\|_\lambda)J_{\alpha,\lambda}'(u_n)=o_n(1)$ strongly in $E_\lambda^*$, it is easy to see from $\|u_n\|_\lambda^2=A+o_n(1)$ and $u_n=u_{\alpha,\lambda}+o_n(1)$ weakly in $E_\lambda$ that $\langle I_{\alpha, \lambda}'(u_n), u_n-u_{\alpha,\beta}\rangle_{E_\lambda^*, E_\lambda}=o_n(1)$ and $I_{\alpha, \lambda}'(u_n)=o_n(1)$ strongly in $E_\lambda^*$, so that $I_{\alpha, \lambda}'(u_{\alpha, \lambda})=0$ in $E_\lambda^*$.  In particular, $\langle I_{\alpha, \lambda}'(u_{\alpha, \lambda}), u_n-u_{\alpha,\beta}\rangle_{E_\lambda^*, E_\lambda}=0$.  Now, we can obtain that
\begin{eqnarray*}
o_n(1)&=&\langle I_{\alpha, \lambda}'(u_n)-I_{\alpha, \lambda}'(u_{\alpha, \lambda}), u_n-u_{\alpha,\beta}\rangle_{E_\lambda^*, E_\lambda}\\
&=&\alpha A\|u_n-u_{\alpha,\beta}\|_{L^2(\bbr^3)}^2+\|u_n-u_{\alpha,\beta}\|_\lambda^2-\mathcal{D}_\lambda(u_n-u_{\alpha,\beta}, u_n-u_{\alpha,\beta})-\|u_n-u_{\alpha,\beta}\|_{L^p(\bbr^3)}^p.
\end{eqnarray*}
Since $u_n-u_{\alpha,\beta}=o_n(1)$ weakly in $E_\lambda$, by using similar arguments in the proofs of \eqref{eq8888} and \eqref{eq8889}, we can see that $u_n-u_{\alpha,\beta}=o_n(1)$ strongly in $E_\lambda$ for $\lambda$ sufficiently large, say $\lambda>\overline{\Lambda}_{k_0^*}$.  Thus, we must have that $J_{\alpha,\beta}'(u_{\alpha,\beta})=0$ in $E_\lambda^*$ for $\lambda>\overline{\Lambda}_{k_0^*}$.
\end{proof}

The following lemma will give a description on the concentration behavior of the nontrivial solutions $u_{\alpha,\lambda}$ as $\lambda\to+\infty$.
\begin{lemma}\label{lem010}
Suppose that the conditions $(A_1)$--$(A_3)$ hold.  For every $\alpha>0$, if either $a_0\geq0$ or $a_0<0$ with $\beta_{k_0^*-1}<1$ then we have  $u_{\alpha,\lambda}\to u_{\alpha}$ strongly in $H^1(\bbr^3)$ as $\lambda\to+\infty$ up to a subsequence.  Furthermore, $u_\alpha$ is a nontrivial solution of $(\mathcal{P}_\alpha^*)$.
\end{lemma}
\begin{proof}
Let $u_{\alpha,\lambda_n}$ be the nontrivial solution obtained in Lemma~\ref{lem009} with $\lambda_n\to+\infty$ as $n\to\infty$.  By Lemma~\ref{lem008}, we can see that
\begin{eqnarray*}
\int_{\bbr^3}(|\nabla u_{\alpha, \lambda_n}|^2+(\lambda_n a(x)+a_0)^+|u_{\alpha,\lambda_n}|^2)dx\leq C_1\quad\text{for all }n\in\bbn.
\end{eqnarray*}
It follows that $\{u_{\alpha, \lambda_n}\}$ is bounded in $D^{1,2}(\bbr^3)$ for $n$ and
\begin{eqnarray*}
\int_{\bbr^3}(a(x)+\frac{a_0}{\lambda_n})^+|u_{\alpha,\lambda_n}|^2dx=o_n(1).
\end{eqnarray*}
Without loss of generality, we may assume that $u_{\alpha, \lambda_n}=u_\alpha+o_n(1)$ weakly in $D^{1,2}(\bbr^3)$.  Thanks to the Sobolev embedding theorem and the conditions $(A_1)$--$(A_3)$, we can see that $u_{\alpha, \lambda_n}=u_\alpha+o_n(1)$ strongly in $L^2(\bbr^3)$ and $u_\alpha\in H_0^1(\Omega)$ with $u_\alpha\equiv0$ on $\bbr^3\backslash\Omega$.  Therefore, by the H\"older and the Sobolev inequality, we get
\begin{eqnarray*}
\|u_{\alpha, \lambda_n}-u_\alpha\|_{L^p(\bbr^3)}\leq\|u_{\alpha, \lambda_n}-u_\alpha\|_{L^2(\bbr^3)}^{\frac{6-p}{2p}}(\|u_{\alpha, \lambda_n}\|_{L^6(\bbr^3)}+\|u_{\alpha}\|_{L^6(\bbr^3)})^{\frac{3p-6}{2p}}=o_n(1).
\end{eqnarray*}
On the other hand, by a variant of the Lebesgue dominated convergence theorem (cf. \cite[Theorem~2.2]{PK74}) and the condition $(A_1)$, we also have $\mathcal{D}_{\lambda_n}(u_{\alpha, \lambda_n}-u_\alpha, u_{\alpha, \lambda_n}-u_\alpha)=o_n(1)$.  Therefore,
\begin{eqnarray*}
\int_{\Omega}|u_\alpha|^pdx&=&\|u_{\alpha, \lambda_n}\|_{L^p(\bbr^3)}^p+o_n(1)\\
&=&\mathcal{D}_{\lambda_n}(u_{\alpha, \lambda_n}, u_{\alpha, \lambda_n})+\|u_{\alpha, \lambda_n}\|_{\lambda_n}^2+\alpha\|\nabla u_{\alpha, \lambda_n}\|_{L^2(\bbr^3)}^4\\
&\geq&\int_{\Omega}\alpha|\nabla u_\alpha|^4+|\nabla u_\alpha|^2+a_0|u_\alpha|^2dx+o_n(1).
\end{eqnarray*}
Note that $u_\alpha\in H_0^1(\Omega)\subset\h$, it is easy to see from $J_{\alpha,\lambda_n}'(u_{\alpha, \lambda_n})=0$ in $E_{\lambda_n}^*$ that $u_\alpha$ is a solution of $(\mathcal{P}_\alpha^*)$.  In particular,
\begin{eqnarray*}
\int_{\Omega}\alpha|\nabla u_\alpha|^4+|\nabla u_\alpha|^2+a_0|u_\alpha|^2dx=\int_{\Omega}|u_\alpha|^pdx.
\end{eqnarray*}
Thus, $u_{\alpha, \lambda_n}=u_\alpha+o_n(1)$ strongly in $D^{1,2}(\bbr^3)$ and
\begin{eqnarray*}
\int_{\bbr^3}\lambda_na(x)u_{\alpha,\lambda_n}^2dx=o_n(1).
\end{eqnarray*}
It follows that $u_{\alpha, \lambda_n}=u_\alpha+o_n(1)$ strongly in $\h$.  Thanks to $c_{\alpha, \lambda}\geq C_\alpha>0$, $u_{\alpha}$ must be nonzero.  Hence, $u_\alpha$ is a nontrivial solution of $(\mathcal{P}_\alpha^*)$.
\end{proof}

Now, we can give the proof of Theorem~\ref{thm001}.

\medskip\par\noindent{\bf Proof of Theorem~\ref{thm001}:}\quad  It follows immediately from Lemmas~\ref{lem009} and \ref{lem010}.
\qquad\raisebox{-0.5mm}{\rule{1.5mm}{4mm}}\vspace{6pt}

\section{Acknowledgements}
Y. Wu is supported by the Fundamental Research Funds for the Central Universities (2014QNA67).


\begin{thebibliography}{999}
\bibitem{A12}
A. Azzollini, The Kirchhoff equation in $\bbr^3$ perturbed by a local
nonlinearity,  {\it Differential Integral Equations,} {\bf 25} (2012), 543--554.



\bibitem{AF12}
C. Alves, G. Figueiredo, Nonlinear perturbations of a periodic Kirchhoff equation in $\bbr^3$, {\it Nonlinear Anal. TMA,} {\bf 75} (2012), 2750--2759.

\bibitem{BW95}
T. Bartsch, Z.-Q. Wang, Existence and multiplicity results for superlinear elliptic problems on $\bbr^3$,   {\it  Comm. Partial Differential Equations, }  {\bf20}(1995),  1725--1741.




\bibitem{BT13}
T. Bartsch, Z. Tang, Multibump solutions of nonlinear Schr\"odinger equations with steep potential well and indefinite potential,  {\it  Discrete Contin. Dyn. Syst., }  {\bf33}(2013),  7--26.


\bibitem{CKW11}
C. Chen, Y. Kuo, T. Wu, The Nehari manifold for a Kirchhoff type problem involving
sign-changing weight functions, {\it J. Differential Equations,} {\bf 250} (2011), 1876--1908.


\bibitem{CWL12}
B. Cheng, X. Wu, J. Liu, Multiple solutions for a class of Kirchhoff type problems with concave nonlinearity, {\it Nonlinear Differ. Equ. Appl.,} {\bf 19} (2012), 521--537.


\bibitem{DS07}
Y. Ding, A. Szulkin, Bound states for semilinear Schr\"odinger equations with sign-changing potential, {\it Calc. Var. Partial
Differential Equations,} {\bf29} (2007), 397--419.


\bibitem{FSX10}
M. Furtado, E. Silva, M. Xavier, Multiplicity and concentration of solutions for elliptic systems with vanishing potentials,   {\it  J. Differential Equations, }  {\bf249}(2010),  2377--2396.


\bibitem{GT121}
Y. Guo, Z. Tang, Multibump bound states for quasilinear
Schr\"odinger systems with critical
frequency,   {\it  J. Fixed Point Theory Appl., } {\bf12}(2012),  135--174.

\bibitem{HLP14}
Y. He, G. Li, S. Peng, Concentrating bound states for Kirchhoff type problems in $\bbr^3$ involving critical Sobolev exponents, {\it Adv. Nonlinear Stud.,} {\bf14} (2014), 441--468.

\bibitem{HZ12}
X. He, W. Zou, Existence and concentration behavior of positive solutions
for a Kirchhoff equation in $\bbr^3$, {\it J. Differential Equations,} {\bf 252} (2012), 1813--1834.

\bibitem{K83}
G. Kirchhoff, Mechanik. Teubner, Leipzig (1883).



\bibitem{LHL11}
X. Liu, Y. Huang, J. Liu, Sign-changing solutions for an asymptotically linear Schr\"odinger equations with deepening potential well,    {\it  Adv. Differential Equations, }  {\bf16}(2011), 1--30.

\bibitem{LLS12}
Y. Li, F. Li, J. Shi, Existence of a positive solution to Kirchhoff type problems
without compactness conditions, {\it J. Differential Equations,} {\bf 253} (2012), 2285--2294.

\bibitem{LLS14}
Z. Liang, F. Li, J. Shi, Positive solutions to Kirchhoff type equations with nonlinearity
having prescribed asymptotic behavior, {\it Ann. Inst. H. Poincar\'e Anal. Non Lin\'eaire,} {\bf31}(2014), 155--167.



\bibitem{LY131}
G. Li, H. Ye, Existence of positive ground state solutions for the nonlinear Kirchhoff type equations in $\bbr^3$, {\it J. Differential Equations,} {\bf 257} (2014), 566--600.


\bibitem{LLT15}
C. Lei, J. Liao, C. Tang, Multiple  positive solutions for Kirchhoff type of problems with singularity and critical exponents, {\it J. Math. Anal. Appl.,} {\bf421} (2015), 521--538.

\bibitem{JZ11}
Y. Jiang, H.-S. Zhou, Schr\"odinger-Poisson system with steep potential well,  {\it  J. Differential Equations, } {\bf251} (2011), 582--608.

\bibitem{PK74}
B. Panda and O. Kapoor, On equidistant sets in normed linear spaces, {\it Bull. Austral. Math.
Soc.,} {\bf11} (1974), 443-454.

\bibitem{ST09}
Y. Sato, K. Tanaka, Sign-changing multi-bump solutions nonlinear Schr\"odinger equations with steep potential wells,   {\it  Trans. Amer. Math. Soc., } {\bf 361}(2009),  6205--6253.

\bibitem{N14}
D. Naimen, Positive solutions of Kirchhoff type elliptic equations involving a critical Sobolev exponent, {\it Nonlinear Differ. Equ. Appl.,} {\bf 21} (2014), 885--914.


\bibitem{N141}
D. Naimen, The critical problem of Kirchhoff type elliptic equations
in dimension four, {\it  J. Differential Equations, } {\bf257} (2014), 1168--1193.



\bibitem{SW14}
J. Sun, T.-F. Wu, Ground state solutions for an indefinite Kirchhoff type problem with steep potential well,   {\it  J. Differential Equations, }  {\bf 256}(2014),  1771--1792.

%
%
%

\bibitem{WZ09}
Z. Wang, H.-S. Zhou, Positive solutions for nonlinear Schr\"odinger equations with deepening potential well,
 {\it J. Eur.
Math. Soc., }  {\bf11}(2009),  545--573.



\bibitem{WHL15}
Y. Wu, Y. Huang, Z. Liu, On a Kirchhoff type problem in $\bbr^N$, {\it J. Math. Anal. Appl.,} {\bf425} (2015), 548--564.

%

\bibitem{YT14}
Y. Ye, C. Tang. Existence and multiplicity of solutions for Schr%
\"{o}dinger-Poisson equations with sign-changing potential,
{\it Calc. Var. Partial Differential Equations,} {\bf 53}(2015), 383--411.

\bibitem{ZP06}
Z. Zhang, K. Perera, Sign changing solutions of Kirchhoff type problems via invariant sets of descent flow, {\it J. Math. Anal. Appl.,} {\bf 317} (2006), 456--463.

\bibitem{ZLZ13}
L. Zhao, H. Liu, F. Zhao, Existence and concentration of solutions for the Schr\"odinger-Poisson equations with steep
well potential,    {\it  J. Differential Equations, }  {\bf255}(2013), 1--23.
\end{thebibliography}
\end{document}